\documentclass{amsart}

\usepackage{amsmath}
\usepackage{amssymb}
\usepackage{rotating}
\usepackage{graphicx}
\usepackage[tableposition=below]{caption}
\usepackage{bm}

\DeclareGraphicsExtensions{.png,.pdf,.eps}
\usepackage[all,2cell,cmtip]{xy}
\usepackage{tikz}
\usepackage{color}
\usepackage{longtable}
\usepackage{soul}
\usepackage{float}

\newtheorem{theorem}{Theorem}[section]

\newtheorem{proposition}[theorem]{Proposition}
\newtheorem{conjecture}[theorem]{Conjecture}
\newtheorem{claim}[theorem]{Claim}

\theoremstyle{definition}

\newtheorem{definition}[theorem]{Definition}

\newtheorem{remark}[theorem]{Remark}

\def\P{{\mathbb P}}

\def\R{{\mathbb R}}
\def\Z{{\mathbb Z}}

\def\cE{{\mathcal E}}

\def\cM{{\mathcal M}}

\def\cO{{\mathcal{O}}}

\def\cU{{\mathcal U}}

\def\rat{\dashrightarrow}

\def\cOperatorname#1{\mathop{\rm #1}\nolimits}

\def\codim{\cOperatorname{codim}}

\def\deg{\cOperatorname{deg}}

\def\rat{\cOperatorname{RatCurves}}

\def\ME{{\cOperatorname{ME}}}

\newcommand{\cME}[1]{\cOverline{\ME}}

\makeindex

\begin{document}

\title{Fano varieties of middle pseudoindex}

\author{Kiwamu Watanabe}
\date{\today}
\address{Department of Mathematics, Faculty of Science and Engineering, Chuo University.
1-13-27 Kasuga, Bunkyo-ku, Tokyo 112-8551, Japan}
\email{watanabe@math.chuo-u.ac.jp}
\thanks{The author is partially supported by JSPS KAKENHI Grant Number 21K03170.}


\begin{abstract} Let $X$ be a complex smooth Fano variety of dimension $n$. In this paper, we give a classification of such $X$ when the pseudoindex is equal to $\dfrac{\dim X+1}{2}$ and the Picard number greater than one. 
\end{abstract}

\maketitle

\section{Introduction} Let $X$ be an $n$-dimensional smooth Fano variety, which is a complex smooth projective variety endowed with an ample anticanonical divisor $-K_X$. The {\it index} of $X$ is defined as 
$$ 
i_X:=\max\{m\in\Z_{>0}\mid -K_X=mL\,\,\mbox{for~some}\,\,L\in {\rm Pic} (X)\}.
$$  
Mukai formulated a conjecture concerning the index $i_X$ and the Picard number $\rho_X$:

\begin{conjecture}[{Mukai conjecture~\cite[Conjecture~4]{Mukai88}}]\label{conj:Mukai} We have $\rho_X(i_X-1)\leq n$, with equality if and only if $X$ is isomorphic to $(\P^{i_X-1})^{\rho_X}$.
\end{conjecture}
As a specific case of this conjecture, Mukai also conjectured that if $2i_X$ is at least $n+2$, then $\rho_X$ is one unless $X$ is isomorphic to $(\P^{i_X-1})^2$ \cite[Conjecture~4']{Mukai88}.
To prove this, Wi\'sniewski \cite{Wis90b} introduced the notion of {\it pseudoindex} $\iota_X$ of $X$:
$$
{\iota}_X:=\min\{-K_X\cdot C\mid C\subset X~\mbox{is~a~rational~curve}\},
$$ 
and arrived at the following theorem:
\begin{theorem}[{\cite{Wis90b, Wis91ind}}]\label{them:Wis} For an $n$-dimensional smooth Fano variety $X$, the following statements hold:
\begin{enumerate}
\item If $2\iota_X>n+2$, then $\rho_X=1$.
\item If $2i_X=n+2$ and $\rho_X>1$, then $X$ is isomorphic to $(\P^{i_X-1})^2$.
\item If $2i_X=n+1$ and $\rho_X>1$, then $X$ is isomorphic to one of the following:
$$\P(\cO_{\P^{i_X}}(2)\oplus\cO_{\P^{i_X}}(1)^{\oplus i_X-1}),\quad \P^{i_X-1}\times Q^{i_X}\quad  \mbox{or} \quad \P(T_{\P^{i_X}}).$$
\end{enumerate}
Here, $Q^m$ denotes a smooth quadric hypersurface, and $T_{\P^{m}}$ represents the tangent bundle of $\P^m$.
\end{theorem}

Theorem~\ref{them:Wis}~(i) and (ii) provide an affirmative answer to \cite[Conjecture~4']{Mukai88}. Following this, by substituting the index $i_X$ with the pseudoindex $\iota_X$, Bonavero-Casagrande-Debarre-Druel proposed a generalized version of Conjecture~\ref{conj:Mukai}:
\begin{conjecture}[{generalized Mukai conjecture~\cite[Conjecture]{BCDD03}}]\label{conj:GMukai} We have $\rho_X({\iota}_X-1)\leq n$, with equality if and only if $X$ is isomorphic to $(\P^{{\iota}_X-1})^{\rho_X}$.
\end{conjecture}

Additionally, Occhetta established in \cite[Corollary~4.3]{Occ06} that a smooth Fano variety $X$ is isomorphic to $(\P^{\iota_X-1})^2$ if $2\iota_X$ equals $n+2$ and $\rho_X$ is greater than one. By combining this with Theorem~\ref{them:Wis}~(i), a generalized version of \cite[Conjecture~4']{Mukai88} was derived. The aim of this brief paper is to establish the following theorem as an extension of Theorem~\ref{them:Wis}~(iii):
\begin{theorem}\label{them:main} Let $X$ be a smooth Fano variety of dimension $n$. If $2\iota_X= n+1$ and $\rho_X>1$, then $X$ is isomorphic to one of the following:
\begin{enumerate}
\item the blow-up of projective space $\P^n$ along a linear subspace $\P^{\iota_X-2}$, i.e., $\P(\cO_{\P^{\iota_X}}(2)\oplus\cO_{\P^{\iota_X}}(1)^{\oplus \iota_X-1});$
\item the product of projective space $\P^{\iota_X-1}$ and a quadric hypersurface $Q^{\iota_X}$, i.e., $\P^{\iota_X-1}\times Q^{\iota_X};$ 
\item the projectivization of the tangent bundle $T_{\P^{\iota_X}}$ of $\P^{\iota_X}$, i.e., $\P(T_{\P^{\iota_X}});$ 
\item the product of projective spaces $\P^{\iota_X-1}$ and $\P^{\iota_X}$, i.e., $\P^{\iota_X-1}\times \P^{\iota_X}.$ 
\end{enumerate}
\end{theorem}
In a manner akin to the argument presented in \cite{Wis91ind}, the pivotal aspect of establishing Theorem~\ref{them:main} lies in proving that $X$ possesses a projective bundle structure $\pi: X\to W$. Subsequently, we prove that the base variety $W$ is either a projective space or a smooth quadric hypersurface. Applying \cite[Corollary~4.7]{Fuj16} and  \cite[Lemme~2.5]{BCDD03} (detailed in Proposition~\ref{prop:bundle} below), we thereby derive our desired conclusion.

\subsection*{Notation and Conventions}\label{subsec:NC} In this paper, we work over the complex number field. Our notation is consistent with the books \cite{Har}, \cite{Kb} and \cite{KM}.  
\begin{itemize}
\item For projective varieties $X, Y$ and $F$, a smooth surjective morphism $f:X\to Y$ is called an {\it $F$-bundle} if any fiber of $f$ is isomorphic to $F$. A surjective morphism $f:X\to Y$ with connected fibers is called an {\it $F$-fibration} if general fibers are isomorphic to $F$.
\item A contraction of an extremal ray is called an {\it elementary contraction}.
\item For a smooth projective variety $X$, we denote by $\rho_X$ the Picard number of $X$ and by $T_X$ the tangent bundle of $X$.
\end{itemize}

\section{Preliminaries} 

\subsection{Fano varieties with large pseudoindex}
Let us start by reviewing certain results concerning Fano varieties with large pseudoindex.
\begin{theorem}[{\cite{CMSB}, \cite{Keb02}, \cite{DH17}}]\label{them:fano:collect} Let $X$ be a smooth Fano variety of dimension $n$ with pseudoindex $\iota_X$. Then, the following holds.
\begin{enumerate}
\item If $\iota_X\geq n+1$, then $X$ is isomorphic to $\P^n$. 
\item If $\iota_X= n$, then $X$ is isomorphic to $Q^n$.
\end{enumerate}
\end{theorem}
The next critical proposition contributes significantly to the proof of Theorem~\ref{them:main}:
\begin{proposition}\label{prop:bundle} Let $X$ be an $n$-dimensional smooth Fano variety with pseudoindex $\iota_X=\dfrac{n+1}{2}$. Assume $X$ admits either a $\P^{\frac{n+1}{2}}$-bundle structure $\pi: X\to W$ or a $\P^{\frac{n-1}{2}}$-bundle structure $\pi: X\to W$. Then $X$ is isomorphic to one of the following:
$$\P(\cO_{\P^{\iota_X}}(2)\oplus\cO_{\P^{\iota_X}}(1)^{\oplus \iota_X-1}),\quad \P^{\iota_X-1}\times Q^{\iota_X},\quad \P(T_{\P^{\iota_X}}), \quad
\P^{\iota_X-1}\times \P^{\iota_X}.$$
\end{proposition}

\begin{proof} By \cite[Lemme~2.5~(a)]{BCDD03}, $W$ is a smooth Fano variety whose pseudoindex is at least $\iota_X=\dfrac{n+1}{2}$. Applying Theorem~\ref{them:fano:collect}, $W$ is isomorphic to $\P^{\frac{n-1}{2}}$, $\P^{\frac{n+1}{2}}$ or $Q^{\frac{n+1}{2}}$. By \cite[Proposition~4.3]{Fuj16}, there exists a vector bundle $\cE$ over $W$ such that $X\cong \P(\cE)$. When $\pi: X\to W$ is a $\P^{\frac{n+1}{2}}$-bundle, $W$ is isomorphic to $\P^{\frac{n-1}{2}}$. In this case, \cite[Lemme~2.5~(c)]{BCDD03} and \cite{
Sato76} tell us that $X$ is isomorphic to $\P^{\iota_X-1}\times \P^{\iota_X}$. When $\pi: X\to W$ is a $\P^{\frac{n-1}{2}}$-bundle, our assertion is derived from \cite[Corollary~4.7]{Fuj16}. Thus, our assertion holds.
\end{proof}

\subsection{Extremal contractions} Extremal contractions play a pivotal role in the study of Fano varieties.  Here, we gather some results concerning extremal rays and extremal contractions.
\begin{definition} For a smooth projective variety $X$ and its $K_X$-negative extremal ray $R\subset \overline{NE}(X)$, the {\it length} of $R$ is defined as $$\ell(R):=\min\{-K_X\cdot C\mid C~\mbox{is a rational curve and}~[C]\in R\}.$$
\end{definition}
\begin{theorem}[{Ionescu-Wi\'sniewski inequality \cite[Theorem~0.4]{Ion86}, \cite[Theorem~1.1]{Wis91} }]\label{them:Ion:Wis} Let $X$ be a smooth projective variety, and let $\varphi: X \to Y$ be a contraction of a $K_X$-negative extremal ray $R$, with $E$ representing its exceptional locus. Additionally, consider $F$ as an irreducible component of a non-trivial fiber of $\varphi$. Then
\begin{eqnarray}\label{IW:inequal} \nonumber
\dim E + \dim F \geq  \dim X + \ell(R)- 1.
\end{eqnarray}
\end{theorem}

\begin{theorem}[{\cite[Theorem~1.3]{HNov13}}]\label{them:HN13} Let $X$ be a smooth projective variety, and let $\varphi: X \to Y$ be a contraction of an  extremal ray $R$ such that any fiber has dimension $d$ and $\ell(R)=d+1$. Then, $\varphi$ is a projective bundle.  
\end{theorem}

\begin{theorem}[{\cite[Theorem~5.1]{AO02}}]\label{them:AO02} For a smooth projective variety $X$ of dimension $n$, the following are equivalent:
\begin{enumerate}
\item There exists an extremal ray $R$ such that the contraction associated to $R$ is divisorial and the fibers have dimension $\ell(R)$.
\item $X$ is the blow-up of a smooth projective variety $X'$ along a smooth subvariety of codimension $\ell(R)+1$.
\end{enumerate} 
\end{theorem}

\begin{remark} For a smooth projective variety $X$, let $\varphi: X\to Y$ and $\psi: X\to Z$ be different elementary contractions of $X$. Then the fibers of $\varphi$ and $\psi$ have a finite intersection. We use this property several times in this paper.
\end{remark}

\subsection{Families of Rational Curves}
Let $X$ denote a smooth projective variety, and let us consider the space of rational curves $\rat^n(X)$  (for details, see \cite[Section~II.2]{Kb}). A {\it family of rational curves} $\mathcal{M}$ on $X$ refers to an irreducible component of $\rat^n(X)$. This family $\mathcal{M}$ is equipped with a $\mathbb{P}^1$-bundle $p: \mathcal{U} \to \mathcal{M}$ and an evaluation morphism $q: \mathcal{U} \to X$. The union of all curves parametrized by $\mathcal{M}$ is denoted by $\text{Locus}(\mathcal{M})$. For a point $x \in X$, the normalization of $p(q^{-1}(x))$ is denoted by $\mathcal{M}_x$, and $\text{Locus}(\mathcal{M}_x)$ denotes the union of all curves parametrized by $\mathcal{M}_x$.

A {\it dominating family} (resp. {\it covering family}) $\mathcal{M}$ is one where the evaluation morphism $q: \mathcal{U} \to X$ is dominant (or surjective). The family $\mathcal{M}$ is termed a {\it minimal rational component} if it is a dominating family with the minimal anticanonical degree among dominating families of rational curves on $X$. Additionally, $\mathcal{M}$ is called {\it locally unsplit} if for a general point $x\in \text{Locus}(\mathcal{M})$, $\mathcal{M}_x$ is proper. The family $\mathcal{M}$ is called {\it unsplit} if it is proper.

\begin{theorem}[{\cite{CMSB, Keb02}}]\label{them:CMSB} Let $X$ be an $n$-dimensional smooth Fano variety and $\cM$ a locally unsplit dominating family of rational curves on $X$. If the anticanonical degree of $\cM$ is at least $n+1$, then $X$ is isomorphic to $\P^n$.
\end{theorem}

\begin{proposition}[{\cite[IV Corollary~2.6]{Kb}}]\label{prop:Ion:Wis:2} Let $X$ be a smooth projective variety and $\cM$ a locally unsplit family of rational curves on $X$. For a general point $x \in {\rm Locus}(\cM)$, 
$$\dim {\rm Locus}(\cM_x) \geq \deg_{(-K_X)}\cM+\codim_X{\rm Locus}(\cM) -1.
$$
Moreover, if $\cM$ is unsplit, this inequality holds for any point $x \in {\rm Locus}(\cM)$. 
\end{proposition}

\section{Proof of the main theorem} 

\subsection{The case when $X$ admits a birational elementary contraction} In this subsection, we aim to establish the following proposition:

\begin{proposition}\label{prop:bir} Let $X$ be a smooth Fano variety with $\iota_X=\dfrac{n+1}{2}$ and $\rho_X>1$. Assume there exists a birational contraction $\varphi: X\to Y$ of an extremal ray $R$. Then $X$ is isomorphic to $\P(\cO_{\P^{\iota_X}}(2)\oplus\cO_{\P^{\iota_X}}(1)^{\oplus \iota_X-1})$.
\end{proposition}

To prove this proposition, throughout this subsection, let $X$ be a smooth Fano variety with $\iota_X=\dfrac{n+1}{2}$ and $\rho_X>1$. Assume there exists a birational contraction $\varphi: X\to Y$ of an extremal ray $R$. We denote by $E$ the exceptional locus of $\varphi$ and by $F$ an irreducible component of a nontrivial fiber of $\varphi$. 

\begin{claim}\label{cl:bir} The exceptional locus $E$ forms a divisor, meaning that $\varphi: X\to Y$ is a divisorial contraction.
\end{claim}

\begin{proof} Let us consider a minimal rational component $\mathcal{M}$ on $X$. According to Theorem~\ref{them:CMSB}, the anticanonical degree of $\mathcal{M}$ is at most $n$. Combining with our assumption that $\iota_X=\frac{n+1}{2}$ and \cite[II, Proposition~2.2]{Kb}, it follows that $\mathcal{M}$ is an unsplit covering family. For any $x \in F$, Proposition~\ref{prop:Ion:Wis:2} implies $\dim {\rm Locus}(\mathcal{M}_x)\geq \frac{n-1}{2}$. To establish our assertion, let us assume the contrary, namely $\text{codim}_XE\geq 2$. Then, by Theorem~\ref{them:Ion:Wis}, it follows that $\dim F\geq \frac{n+3}{2}$. Consequently, $\dim({\rm Locus}(\mathcal{M}_x)\cap F)\geq 1$. By \cite[II, Corollary~4.21]{Kb}, this leads to a contradiction.
\end{proof}

Utilizing Theorem~\ref{them:Ion:Wis}, we infer $\dim F\geq \iota_X=\dfrac{n+1}{2}$. Since the Kleiman-Mori cone $\overline{NE}(X)$ of a Fano variety $X$ is polyhedral and each extremal ray is generated by a rational curve, we can identify an extremal ray $R'$ and a rational curve $C'$ such that $R'=\R_{\geq 0}[C']$, $\ell(R')=-K_X\cdot C'$ and $E\cdot C'>0$. We denote by $\psi: X\to Z$ the contraction of an extremal ray $R'$. 

\begin{claim}\label{cl:fiber} $\psi: X\to Z$ is of fiber type. 
\end{claim}

\begin{proof} Assuming the contrary, that is, $\psi$ is of birational type, let $E'$ be the exceptional locus and $F'$ an irreducible component of a nontrivial fiber of $\psi$. By Theorem~\ref{them:Ion:Wis}, we have $\dim F'\geq \iota_X=\dfrac{n+1}{2}$. Since $E\cdot C'>0$, we have $E\cap E'\neq \emptyset$. By replacing the fibers $F$ and $F'$ if necessary, we may assume that $F\cap F'\neq \emptyset$. Then we obtain 
$$
\dim F+\dim F'-\dim X \geq \dfrac{n+1}{2} \times 2 -n =1.
$$ 
This leads to $\varphi=\psi$; this is a contradiction. Therefore, $\psi: X\to Z$ is of fiber type. 
\end{proof}

Let $F_{\rm gen}'$ denote any fiber of $\psi$ whose dimension is equal to $\dim X-\dim Z$.
Applying Theorem~\ref{them:Ion:Wis}, we have 
\begin{eqnarray}\label{eq:1}
\dim F'_{\rm gen}\geq \ell(R')-1\geq \dfrac{n-1}{2}. 
\end{eqnarray}
According to $E\cdot C'>0$, $\psi|_E: E\to Z$ is surjective. Since $\psi|_F: F \to Z$ is finite, we have 
\begin{eqnarray}\label{eq:2}
\dim Z \geq \dim F\geq \dfrac{n+1}{2}. 
\end{eqnarray}
Now we have
$n=\dim X=\dim F'_{\rm gen}+\dim Z \geq  \dfrac{n-1}{2}+\dfrac{n+1}{2}=n$. 
This yields 
$$
\left(\dim F'_{\rm gen},  \dim Z \right)= \left(\dfrac{n-1}{2},\dfrac{n+1}{2}\right)
$$
Moreover (\ref{eq:1}) and (\ref{eq:2}) imply that $\ell(R')=\dim F =\dfrac{n+1}{2}$.
Assume there exists a jumping fiber $F'_{\rm sp}$ of $\psi$, meaning $\dim F'_{\rm sp}>\dim F'_{\rm gen}=\dfrac{n-1}{2}$.
Taking an irreducible component $F$ of a nontrivial fiber of $\varphi$ such that $F'_{\rm sp}\cap F\neq \emptyset$, we have
$$
\dim F'_{\rm sp}+\dim F-\dim X>\dfrac{n-1}{2}+\dfrac{n+1}{2}-n=0.
$$ This is a contradiction. As a consequence, $\psi$ is equidimensional. Since $\ell(R')=\dfrac{n+1}{2}=\dim F'_{\rm gen}+1$, Theorem~\ref{them:HN13} tells us that $\psi$ is a $\P^{\frac{n-1}{2}}$-bundle.

By Theorem~\ref{them:Ion:Wis}, we have 
$$
n-1 +\dfrac{n+1}{2} =\dim E+\dim F \geq n+\ell(R)-1\geq n-1+\dfrac{n+1}{2}.
$$
This yields that, for any nontrivial fiber $F$ of $\varphi$, we have $\dim F=\dfrac{n+1}{2}=\ell(R)$. Applying Theorem~\ref{them:AO02}, we see that $\varphi: X\to Y$ is the blow-up of a smooth variety $Y$ along a smooth subvariety $\varphi(E)$ of codimension $\ell(R)+1=\dfrac{n+3}{2}$. Hence any nontrivial fiber $F$ of $\varphi$ is isomorphic to $\P^{\frac{n+1}{2}}$. Since we have a finite morphism $\psi|_F: F\cong \P^{\frac{n+1}{2}} \to Z$ between smooth projective varieties of dimension $\dfrac{n+1}{2}$, $\psi|_F: F\to Z$ is a finite surjective morphism. By \cite[Theorem~4.1]{Laz84}, $Z$ is isomorphic to $\P^{\frac{n+1}{2}}$. Since $\psi: X\to \P^{\frac{n+1}{2}}$ is a $\P^{\frac{n-1}{2}}$-bundle, Proposition~\ref{prop:bundle} implies Proposition~\ref{prop:bir}.

\subsection{The case when any elementary contraction of $X$ is of fiber type.} In this subsection, we aim to establish the following proposition:

\begin{proposition}\label{prop:fib} Let $X$ be a smooth Fano variety with $\iota_X=\dfrac{n+1}{2}$ and $\rho_X>1$. Assuming that any elementary contraction of $X$ is of fiber type, then $X$ is isomorphic to $\P^{\iota_X-1}\times Q^{\iota_X}$, $\P(T_{\P^{\iota_X}})$, or $\P^{\iota_X-1}\times \P^{\iota_X}$.
\end{proposition}

To prove this proposition, throughout this subsection, we stay within the confines of the present subsection, maintaining the setting where $X$ is a smooth Fano variety with $\iota_X=\dfrac{n+1}{2}$ and $\rho_X>1$. Assume that any elementary contraction of $X$ is of fiber type. For different extremal rays $R$ and $R'$ of $\overline{NE}(X)$, consider the elementary contractions $\varphi:X\to Y$ and $\psi: X\to Z$ associated to $R$ and $R'$ respectively. We denote by $F$ (resp. $F'$) any fiber of $\varphi$ (resp. $\psi$). Using Theorem~\ref{them:Ion:Wis}, we infer
\begin{eqnarray}\label{eq:3}
\dim F\geq \ell(R)-1\geq \dfrac{n-1}{2}\quad \mbox{and}\quad
\dim F'\geq \ell(R')-1\geq \dfrac{n-1}{2}.
\end{eqnarray}
Since we have 
\begin{eqnarray}\label{eq:4}
\dim F+\dim F'-n\leq 0, 
\end{eqnarray}
it turns out that $\dim F$ and $\dim F'$ are at most $\dfrac{n+1}{2}$. Thus, denoting by $F_{\rm gen}$ (resp. $F_{\rm gen}'$) any fiber of $\varphi$ (resp. $\psi$) whose dimension is equal to $\dim X-\dim Y$ (resp. $\dim X- \dim Z$), $(\dim F_{\rm gen }, \dim Y)$ and $(\dim F_{\rm gen}', \dim Z)$ are either:
$$
\left(\dfrac{n+1}{2}, \dfrac{n-1}{2}\right)\quad\mbox{or}\quad \left(\dfrac{n-1}{2}, \dfrac{n+1}{2}\right).
$$
We now claim:
\begin{claim}\label{cl:varphi} $\varphi$ and $\psi$ are one of the following:
\begin{enumerate}
\item a $\P^{\frac{n+1}{2}}$-bundle;
\item a $Q^{\frac{n+1}{2}}$-fibration;
\item a $\P^{\frac{n-1}{2}}$-fibration.
\end{enumerate}
\end{claim}
\begin{proof} It is enough to consider the structure of $\varphi$. 
 Assume $\dim F_{\rm gen }=\dfrac{n+1}{2}$. By inequality (\ref{eq:4}), $\varphi$ is equidimensional. By inequality (\ref{eq:3}), we see that $\ell(R)=\dfrac{n+3}{2}$ or $\dfrac{n+1}{2}$. In the former case, it follows from Theorem~\ref{them:HN13} that $\varphi$ is a $\P^{\frac{n+1}{2}}$-bundle. In the latter case, following Theorem~\ref{them:fano:collect}, $\varphi$ is a $Q^{\frac{n+1}{2}}$-fibration. On the other hand, if $\dim F_{\rm gen }=\dfrac{n-1}{2}$, then Theorem~\ref{them:fano:collect} yields that $\varphi$ is a $\P^{\frac{n-1}{2}}$-fibration.
\end{proof}
Without loss of generality, we may assume that $\dim F_{\rm gen } \geq \dim F_{\rm gen }'$. Then the pair of $\varphi$ and $\psi$ is one of the following:
\begin{enumerate} 
\item[(A)] $\varphi$ is a $\P^{\frac{n+1}{2}}$-bundle and $\psi$ is a $\P^{\frac{n-1}{2}}$-fibration;
\item[(B)] $\varphi$ is a $Q^{\frac{n+1}{2}}$-fibration and $\psi$ is a $\P^{\frac{n-1}{2}}$-fibration;
\item[(C)] $\varphi$ and $\psi$ are $\P^{\frac{n-1}{2}}$-fibrations.
\end{enumerate}
By inequality (\ref{eq:4}) and Theorem~\ref{them:HN13}, in case (B), $\psi$ is a $\P^{\frac{n-1}{2}}$-bundle. In case (C), either $\varphi$ or $\psi$ turns into a $\P^{\frac{n-1}{2}}$-bundle. Consequently, Proposition~\ref{prop:bundle} infers Proposition~\ref{prop:fib}. \subsection{Conclusion}
By Proposition~\ref{prop:bir} and Proposition~\ref{prop:fib}, we obtain Theorem~\ref{them:main}.

\subsection*{Acknowledgments} The author would like to extend their gratitude to Professor Taku Suzuki for reviewing the initial draft of this paper. Professor Suzuki not only identified errors but also provided a proof of Claim $3.2$.

\subsection*{Conflict of Interest.} The author has no conflicts of interest directly relevant to the content of this article.

\subsection*{Data availability.} Data sharing not applicable to this article as no data sets were generated or analyzed during the current study.

\bibliographystyle{plain}
\bibliography{biblio}
\end{document}